\documentclass{amsart}
\usepackage[utf8]{inputenc}

\usepackage{amsmath,amssymb,amsbsy,amsfonts,amsthm,latexsym,mathabx,amsopn,amstext,amsxtra,euscript,amscd,stmaryrd,mathrsfs,cite,array,mathtools,enumerate,enumitem}

\usepackage{hyperref}
\hypersetup{
    colorlinks=true,
    linkcolor=blue,
    filecolor=magenta,      
    urlcolor=cyan,
    citecolor=magenta,
    pdftitle={Smooth polynomials},
    pdfpagemode=FullScreen,
    }

\newtheorem{theorem}{Theorem}
\newtheorem{lemma}[theorem]{Lemma}

\newcommand{\F}{\mathbb{F}}

\newcommand{\CC}{\mathbb{C}}

\newcommand{\Fq}{\mathbb{F}_q}

\newcommand{\mand}{ \quad \text{and} \quad}

  \renewcommand{\a}{\mathbf{a}}
    \newcommand{\s}{\mathbf{s}}
 \newcommand{\rank}{\mathop{\mathrm{rank}}}

 \usepackage{todonotes}

\def\cB{{\mathcal B}}

\def\cG{{\mathcal G}}

\def\cM{{\mathcal M}}

\def\cP{{\mathcal P}}

\def\cR{{\mathcal R}}

\def\bx{{\mathbf x}}
\def\x{{\mathbf x}}
\def\y{{\mathbf y}}
\def\h{{\mathbf h}}
\def\c{{\mathbf c}}

\newcommand{\w}{\mathbf{w}}

\newcommand{\rk}{\mathop{\mathrm{rank}}}

\title[Rudin-Shapiro function along irreducible polynomials]{Rudin-Shapiro function along irreducible polynomials over finite fields}
  \author[L.~M{\'e}rai]{L{\'a}szl{\'o} M{\'e}rai}
  \address{Department of Computer Algebra, Eötvös Loránd University,  Pázmány P. sétány 1/c, H-1117, Budapest, Hungary} 
  \email{merai@inf.elte.hu}

\allowdisplaybreaks

\date{\today}


\begin{document}

\maketitle

\begin{abstract}
    Let $q$ be an odd prime power and  $\Fq$ be the finite field of $q$ elements. We define the Rudin-Shapiro function $R$ on monic polynomials $f=t^n+f_{n-1}t^{n-1}+\dots + f_0\in\Fq[t]$  by 
    $$
R(f)=\sum_{i=1}^{n-1}f_if_{i-1}.
    $$
    We investigate the distribution of the Rudin-Shapiro function along irreducible polynomials. We show that the number of irreducible polynomials $f$ with $R(f)=\gamma$ for any $\gamma\in\Fq$ is asymptotically $q^{n-1}/n$ as $n\rightarrow\infty$.  
\end{abstract}

\section{Introduction}

In recent years, many spectacular results have been obtained on important problems combining some arithmetic properties of integers and some conditions on their digits in a given basis, see for example \cite{Bou13,Bou15,DMR19,MR09,MR15,M16,M18,SW19,Sw20}.
In particular, Mauduit and Rivat \cite{MR15} showed that the \emph{Rudin-Shapiro sequence} along prime numbers is well-distributed.

The Rudin-Shapiro sequence is a classical digit-defined sequence: over the integers it is defined in terms of the parity of the number of occurrences of the block $11$ in the binary expansion. Thus the result of Mauduit and Rivat \cite{MR15} shows that this highly structured digital condition is nevertheless equidistributed along the primes.

A natural question is to study analog problems in finite fields, see for example \cite{DMS15,DS13,DES16,DMW21,Ha16,Mer24+,MW22,M19,Po13,P19,Sw18,Sw18b,Sw18a}.
Many of these problems can be solved for finite fields although their analogs for integers are actually out of reach.

In particular, Dartyge and Sárközy \cite{DS13} investigated the distribution of the \emph{Thue-Morse function} in finite fields over polynomial values (see also \cite{DMW21}) and Dartyge, Mérai and Winterhof \cite{DMW21} studied the frequency problem of the \emph{Rudin-Shapiro function} in finite fields over polynomial values. 

Moreover, Porritt \cite{PoPhD} studied the distribution of linear forms of irreducible polynomials as a generalization of the Thue-Morse function for polynomials.

In this paper, we investigate the frequency problem for the Rudin-Shapiro function along irreducible polynomials. 
In the standard analogy between the integer and function field settings, polynomial coefficients play the role of digits, while monic irreducible polynomials play the role of prime numbers. Thus our main result provides a function field analogue of the equidistribution of the Rudin-Shapiro sequence along the primes.

\bigskip

Let $q$ be an odd prime power and $\Fq$ be the finite field of $q$ elements. For a monic polynomial $f\in\Fq[t]$ of degree $n\geq 2$, define the \emph{Rudin-Shapiro function} $R(f)$ as
\begin{equation*}
R(f)=\sum_{i=1}^{n-1}f_if_{i-1} \quad \text{with } f=t^n+f_{n-1}t^{n-1}+\dots+f_0\in\Fq[t].
\end{equation*}

Let $\cP(n)$ be the set of \emph{monic irreducible} polynomials of degree $n$. It is well-known, that
\begin{equation}\label{eq:PNT}
\frac{q^n}{n}-2\frac{q^{n/2}}{n}\leq \pi_q(n)\leq \frac{q^n}{n}    
\end{equation}
with $\pi_q(n)=\#\cP(n)$
(see for example \cite[Theorem~6.5.1]{BaSha}).
For $\gamma\in\Fq$, let
$$
\cR_n(\gamma)=\{f\in\cP(n): R(f)=\gamma \}.
$$
Our goal is to prove that the size of $\cR_n(\gamma)$ is asymptotically the same for all $\gamma$.

Our main result is the following theorem.

\begin{theorem}\label{thm:1} 
Let $0<\delta<1/38$. Then,
for any odd prime power $q$, every $\gamma\in \Fq$ and every $n\geq 3$ we have
$$
\#\cR_n(\gamma) = \frac{q^{n-1}}{n} +O\left( n^{3/2}q^{(37/38 +\delta)n +11/5}\right),
$$
where the implied constant may depend on $\delta$.
\end{theorem}

We remark, that  for $n=2$ the result follows from the Weil bound. Indeed, for a quadratic monic polynomial $f$, we have $R(f)=\gamma$ if it has the form $f=t^2+\gamma x^{-1} t + x$ for some $x\in\Fq^*$. 
Such a polynomial is irreducible if and only if its discriminant $\mathrm{Disc}(f)=\gamma^2 x^{-2}-4x$ is a quadratic non-residue. 
By the Weil bound, the latter condition holds for
$\frac{q-1}{2} + O(q^{1/2})$ values of $x\in\Fq^*$. For more details see comment before \cite[Problem~8]{MeWi22} in the case where $q$ is prime; the argument extends readily to any odd prime powers.

The proof of Theorem~\ref{thm:1} relies essentially on the assumption that the characteristic of $\Fq$ is odd. The case of finite fields of even characteristic appears to require different ideas and is left as an open problem.

\section{Outline of the proof}

For a  positive integer $n$,  let $\cG(n)\subset \Fq[t]$ be the set of polynomials of degree at most $n$ and $\cM(n)\subset \cG(n)$ be the set of \emph{monic} polynomials of degree $n$.

For $\ell\geq 0$, write
\begin{equation}\label{eq:S-def}
S^{(\ell)}(f)=\sum_{i=\ell}^nf_if_{i-\ell} \quad \text{with }  f=f_nt^n+f_{n-1}t^{n-1}+\dots+f_0\in\Fq[t].    
\end{equation}
For $\ell=1$, we simply write
\begin{equation*}
S(f)=S^{(1)}(f).
\end{equation*}
Clearly, for monic polynomials $f$ we have 
\begin{equation}\label{eq:lin_red}
R(f)=S(f)-f_{n-1}.
\end{equation}
\bigskip

In order to prove Theorem~\ref{thm:1}, we introduce Vaughan's identity for polynomials in Section~\ref{sec:vau}. Then the main tool is to investigate single and double exponential sums in terms of $S(f)$ in Section~\ref{sec:qu}. The bounds of the exponential sum estimates are given in terms of the rank of quadratic forms which is studied in Sections \ref{sec:rev} and \ref{sec:rank}. Then we conclude the proof of Theorem~\ref{thm:1} in Section~\ref{sec:proof}.

\bigskip

For given functions $F$ and $G$, the notations $F\ll G$, $G \gg F$ and $F =O(G)$ are all equivalent to the statement that the inequality $|F| \leq c|G|$
holds with some constant $c > 0$. 

\section{preliminaries}

For a non-zero polynomial $f\in\F_q[t]$ put
$$
|f|=q^{\deg f}
$$
with the convention $|0|=0$.

Let $\tau(f)$ denote the number of monic divisors of $f\in\Fq[t]$. First, we state the following bound on $\tau(f)$.

\begin{lemma}\label{lemma:tau}
For any $\varepsilon>0$ and $f\in\F_q[t]$ of degree $\deg f>1$ we have
$$
\tau(f) \ll \min\left\{ |f|^{(2+\varepsilon)/\log \deg f},2^{\deg f}
\right\},
$$
where the implied constant can depend on $\varepsilon$.
\end{lemma}
\begin{proof}
The first bound follows from \cite[Lemma~16]{Mer24+}, while the second bound follows from the fact that a polynomial $f$ has at most $\deg f$ irreducible divisor. 
\end{proof}

We also have the following bound on the second moment of $\tau$, see \cite[Lemma~8]{BiLe}.

\begin{lemma}\label{lemma:tau_moment}
    We have
    $$
\sum_{f\in \cM(n)}\tau(f)^2\leq 4n^3 q^n.
    $$
\end{lemma}

\section{Auxiliary results}
\subsection{Exponential sums of quadratic forms of polynomials}\label{sec:qu}

Let $n$ be a positive integer and let $Q$ be a quadratic form on $\Fq^n$. Then, $Q$ can be written as $Q(\bx)=\bx^TM\bx$ for some symmetric matrix $M\in\Fq^{n\times n}$. The \emph{rank} of $Q$ is the rank of the matrix $M$. The first result we need is the following bound. For completeness, we also include the proof.   

\begin{lemma}\label{lemma:gauss}
Let $q$ be an odd prime power and let $Q$ be a quadratic form, and let $L$ be a linear form on $\Fq^n$. Assume that $Q$ has rank at least $r$. Then
$$
\left|\sum_{\x\in \Fq^n}\psi(Q(\x)+L(\x))\right|\leq q^{n-r/2}
$$
for any non-trivial additive character $\psi$ of $\Fq$.
\end{lemma}

\begin{proof}
Write $Q(\x)=\x^TM\x$ for some symmetric matrix $M\in\Fq^{n\times n}$ of rank at least~$r$ and let $B$ be the corresponding bilinear form 
$$
B(\x,\y)=\x^TM\y.
$$
Write also $B_\x(\y)=B(\x,\y)$. Let $K$ be the subspace of vectors $\x\in\Fq^n$ such that $B_\x=0$. It's codimension is the rank of $M$ and thus of the quadratic form $Q$.

We have
\begin{align*}
&\left|\sum_{\x\in \Fq^n}\psi(Q(\x)+L(\x))\right|^2\\
&=\sum_{\x,\h\in \Fq^n}\psi(Q(\x+\h)+L(\x+\h))\overline{\psi(Q(\x)+L(\x))}\\
&=\sum_{\h\in \Fq^n}\psi(Q(\h)+L(\h)) 
\sum_{\x\in \Fq^n}
\psi(2B_\h(\x)).
\end{align*}
If $\h\not \in K$, then $B_\h$ is a non-zero linear form and takes all the values of $\Fq$ with frequency $q^{n-1}$,  thus 
$$
\sum_{\x\in \Fq^n}
\psi(2B_h(\x))=q^{n-1}\sum_{x\in \Fq}
\psi(x)=0.
$$
For $\h\in K$, we have
$$
\sum_{\x\in \Fq^n}
\psi(2B_\h(\x))=q^n.
$$
The assertion follows by noticing $\# K=q^{n-r}$.
\end{proof}

\subsection{Reversed polynomials}\label{sec:rev}
For a polynomial $a\in\Fq[t]$ of degree $n$ of the form $a=a_nt^n+\dots + a_1t+a_0$ let $a^*$ be the polynomial with reversed coefficients, that is,
$$
a^*(t)=t^na(t^{-1})=a_0t^n+\dots + a_{n-1}t+a_n\in\Fq[t].
$$
First we show, that the quantities $S^{(\ell)}(a)$ can be determined by using the reversed polynomial $a^*$ of $a$.

\begin{lemma}\label{lemma:star_coeff}
    Let $n\geq 1$. For any polynomial $a\in\cG(n)$, we have 
    $$
    a^*\cdot a=S^{(0)}(a)t^n+\sum_{\ell=1}^n S^{(\ell)}(a)\left(t^{n+\ell}+t^{n-\ell}\right). 
    $$
\end{lemma}
\begin{proof}
The coefficient of the term $t^{n-\ell}$ in $a^*\cdot a$ is
$$
\sum_{ i+j=n-\ell} a_{n-i}a_{j}=\sum_{j=0}^{n-\ell} a_{j+\ell}a_{j}=S^{(\ell)}(a).
$$
Moreover, the coefficient of $t^{n-\ell}$ is the same as the coefficient of $t^{n+\ell}$.
\end{proof}

Next, we investigate the following type of quadratic equations in terms of the polynomials $a$ and its reversed $a^*$.
\begin{lemma}\label{lemma:quadratic_equation}
Let $\varepsilon >0$.
For given integer $n\geq 2$, and polynomial $f\in \cG(2n)$, let $N(f)$ be the number of polynomials $a\in\cG(n)$ with 
\begin{equation}\label{eq:quadratic_equation}
a^*\cdot a=f.
\end{equation}
Then
$$
N(f)\ll \min\left\{ 2^n,q^{(2+\varepsilon)\frac{n}{\log n}}\right\},
$$
where the implied constant may depend on $\varepsilon$.
\end{lemma}

\begin{proof} Clearly, for any polynomials $a,b \in \Fq[t]$, we have
\begin{equation}\label{eq:star}
    (a\cdot b)^*=a^* \cdot b^*.
\end{equation}

Now assume that $f$ has the form $$
f=b^*\cdot b \quad \text{for some } b\in \cG(n),
$$
since otherwise the bound trivially holds.

If $a$ is a polynomial satisfying \eqref{eq:quadratic_equation}, put
$$
d=\gcd(a,b)
$$
with monic polynomial $d$.
Then by \eqref{eq:star}
$$
f= d \cdot \frac{b}{d} \cdot d^* \cdot \left(\frac{b}{d}\right)^*,
$$
and as
$$
\gcd(a/d,b/d)=1 \mand \gcd((a/d)^*,(b/d)^*)=1 
$$
we must have 
$$
a=d \cdot \left(\frac{b}{d}\right)^*. 
$$
That is, the divisor $d$ uniquely determines $a$. Whence, the number of polynomials $a$ with \eqref{eq:quadratic_equation} is at most $\tau(b)$. Then the bound follows from Lemma~\ref{lemma:tau}.
\end{proof}

 \subsection{Rank of quadratic forms}\label{sec:rank}
For $k\leq n$ and for a fixed $a\in \cM(k)$ the map $h\mapsto S(ah)$ with $h\in\cG(n-k)$ is a quadratic form in the coefficients of $h$. First we investigate its rank. 

\begin{lemma}\label{lemma:triv-rank}
   Let $a\in\cM(k)$ and for $h\in\cG(n-k)$ consider the map 
\begin{equation}\label{eq:quadratic}
Q_a: h\mapsto S(ah),    
\end{equation}
which is a quadratic form in the coefficients of $h$. For $k\leq n/2$ we have $\rank Q_a\geq n-2k$.
\end{lemma}
\begin{proof}
Let $B\in\F_q^{(n+1)\times (n+1)}$ be defined as
\begin{equation}\label{eq:B-def}
    B_{i,j}=
\left\{
\begin{array}{cl}
  1/2   & \text{if } |i-j|=1, \\
  0   &  \text{otherwise},
\end{array}
\right. \quad 0\leq i,j\leq n,
\end{equation}
and $A(a)\in\F_q^{(n+1)\times (n-k+1)}$ as
\begin{equation}\label{eq:A(a)-def}
    A(a)_{i,j}= a_{i-j}, \quad 0\leq i\leq n, \ 0\leq j\leq n-k
\end{equation}
where we put $a_{\ell}=0$ for $\ell >k$ or $\ell<0$
by convention. Observe, that the minor $$(A(a))_{k\leq i\leq n, 0\leq j\leq n-k}$$  is an upper triangular matrix with $a_k\neq 0$ in the diagonal, thus $\rank A(a)=n-k+1$. 

By the definition of $B$ and $A(a)$, we have 
\begin{equation}\label{eq:Q_a-form}
Q_a(h)=\h^t A(a)^TBA(a) \h    
\end{equation}
where $\h$ is the coefficient vector of $h$, $\h=(h_0,\dots, h_{n-k})$. Whence $\rk Q_a=\rk A(a)^TBA(a)$.

By Sylvester rank inequality, we have for any matrices $M\in\Fq^{m\times n}$ and $N\in\Fq^{n\times k}$ that $\rank MN\geq \rank M +\rank N -n$. Specifically, as $\rank B \geq n$ and $\rank  A(a) =n-k+1$, we get
$$
\rank BA(a) \geq \rank B  +\rank A -n-1\geq n-k
$$
whence
$$
\rank A(a)^TBA(a) \geq \rank A(a)^T  + \rank BA(a)  -n-1\geq n-2k.
$$
\end{proof}

\begin{lemma}\label{lemma:single:bad:a} 
 Let $a\in\cM(k)$ and for $h\in\cG(n-k)$ consider the map $Q_a$ defined by \eqref{eq:quadratic}
For $0\leq r\leq n-k$, the number of $a$ such that $\rank Q_a<r$ is $O(q^{(2k+5)/3+2r})$ if $k\leq (3n+1)/4$ and $O(q^{2k-n+2r+1})$ if $k> (3n+1)/4$.
\end{lemma}
\begin{proof}

By \eqref{eq:Q_a-form}, we have $\rk Q_a=\rk A(a)^TBA(a)$, where $B$ and $A(a)$ are defined as \eqref{eq:B-def} and \eqref{eq:A(a)-def}.

The matrix $A(a)^TBA(a)\in\Fq^{(n-k+1)\times (n-k+1)}$ can be written as
\begin{equation}\label{eq:ABA}
(A(a)^TBA(a))_{i,j}=\frac{1}{2}\left(S^{(|i-j-1|)}(a)+S^{(|j-i-1|)}(a)\right), \quad 0\leq i,j\leq n-k,    
\end{equation}
where $S^{(\ell)}$ is defined as \eqref{eq:S-def}.

In order to estimate the number of polynomials $a$ with 
\begin{equation}\label{eq:ABA-rank}
    \rk A(a)^TBA(a)<r
\end{equation} 
consider the $i$-th rows $\c ^T _i\in\Fq^{n-k+1}$ of $A(a)^TBA(a)$ with $1\leq i\leq r$. By \eqref{eq:ABA-rank} these rows are linearly dependent. In this case, the values of $S^{(\ell)}(a)$ with $0\leq \ell <r$ determine further values of $S^{(m)}$. Indeed, assume that $\c_1,\dots, \c_r$ are linearly dependent, that is, there are $\lambda_1,\dots, \lambda_r\in\Fq$, not all of them are zero, such that 
\begin{equation}\label{eq:dep}
    \lambda_1\c_1+\dots+ \lambda_r\c_r=\mathbf{0}.
\end{equation} Let $1\leq i\leq r$ be the smallest index with $\lambda_i\neq 0$. Then by \eqref{eq:ABA} and \eqref{eq:dep} we get for $r< j\leq n-k$ that
\begin{multline}\label{eq:recursion}
\lambda_i \left(S^{(j-i+1)}(a)+S^{(j-i-1)}(a)\right)+\dots +\lambda_r \left(S^{(j-r+1)}(a)+S^{(j-r-1)}(a)\right)=0.
\end{multline}
That is, the values of $S^{(m)}(a)$ with $r-i +1 < m\leq n-k-i+1$ are determined by $S^{(0)}(a),\dots, S^{(r-i+1)}(a)$.

Let $d$ be a positive integer to be fixed later with
\begin{equation}\label{eq:d}
r-i< d\leq \min\{n-k-i+1,k\}.
\end{equation}
Fix the values $S^{(0)}(a)=s_0,\dots, S^{(r-i+1)}(a)=s_{r-i+1}$ and let $s_{r-i+2}, \dots, s_{d}$ be the values determined by the recursion \eqref{eq:recursion}, that is, 
$$
s_{\ell}=-s_{\ell-2}-\frac{1}{\lambda_i}\left(\lambda_{i+1}(s_{\ell-1}+s_{\ell-3})+\dots + \lambda_r (s_{\ell -r+i}+s_{\ell -r+i-2}) \right) $$
for $r-i+1< \ell\leq d$.

We count
\begin{align*}
N(\s)=\#\{a\in\cG(k):\, & S^{(0)}(a)=s_0,\dots, S^{(d)}(a)=s_{d}\}.
\end{align*}
For $\ell\geq 1$, write $S_\ell\in\Fq^{(k+1)\times (k+1)}$ defined as
\begin{equation*}
    \left(S_{\ell}\right)_{i,j}=
\left\{
\begin{array}{cl}
  1/2   & \text{if } |i-j|=\ell, \\
  0   &  \text{otherwise},
\end{array}
\right. \quad 0\leq i,j\leq k
\end{equation*}
and put $S_0\in\F_q^{(k+1)\times (k+1)}$ as the unit matrix. 
Then $S^{(\ell)}(a)=\a^T S_\ell \a$, where $\a\in\Fq^{k+1}$ is the coefficient vector $\a=(a_0,\dots, a_k)$ of~$a$.
Then, by the orthogonality of characters, we have
\begin{align}\label{eq:N-solutions}
 N(\s) 
  &=\frac{1}{q^{d+1}}\sum_{\a\in \Fq^{k+1}}\sum_{w_0\in\Fq}\psi\left(w_0(\a^T S_0 \a - s_0)\right)\dots   \sum_{w_{d}\in\Fq}\psi\left(w_{d}(\a^T S_{d} \a - s_{d})\right) \notag \\
  &=q^{k-d}+\frac{1}{q^{d+1}}\sum_{\substack{\w\in\Fq^{d+1}\\ \w\neq 0 }} \sum_{\a\in \Fq^{k+1}} \psi(\a^TS(\w)\a -\w^T \cdot \s)
 \end{align}
with $S(\w)=w_0S_0+\dots+w_{d}S_{d}$.
By Lemma~\ref{lemma:gauss}, we have for $\w\neq 0$, that
\begin{equation}\label{eq:N-sloution-sum}
  \sum_{\a\in \Fq^{k+1}} \psi(\a^TS(\w)\a -\w^T \cdot \s)  \ll q^{k+1-\rk S(\w)/2}.
\end{equation}
If 
$\nu\leq  d$ is the maximal index such that $w_\nu\neq 0$, then the minor
$$
(S(\w))_{\nu\leq i\leq k, 0\leq j\leq k-\nu  }\in\Fq^{(k-\nu +1)\times (k-\nu+1)  }
$$
is an upper triangular matrix with $w_\nu\neq 0$ in the diagonal, thus it has rank $k-\nu+1\geq k-d+1$. It yields $\rk S(\w)\geq k-d+1$. Then \eqref{eq:N-solutions} and \eqref{eq:N-sloution-sum} give
$$
N(\s) = q^{k-d} +O\left(   q^{k+1-(k-d+1)/2}\right)=q^{k-d} +O\left(   q^{(k+d+1)/2}\right).
$$

In the recursion \eqref{eq:recursion}
we can choose the initial values $s_0,\dots, s_{r-i+1}$ in $q^{r-i+2}$ ways and the coefficients $\lambda_i,\dots, \lambda_r$ in at most $q^{r-i+1}$ ways. Hence the number of $a$ 
satisfying \eqref{eq:recursion}
is at most
\begin{equation}\label{eq:almost_final}
q^{2r-2i+3}\left(q^{k-d} +O\left(   q^{(k+d+1)/2}\right)\right).
\end{equation}
If $k\leq (3n+1)/4$, we have $(k-1)/3\leq n-k$. Moreover, we can assume that $r\leq (k-5)/6$ as otherwise the bound is trivial. 
The choice 
$$
d=\min\{n-k-i+1, \lceil (k-1)/3 \rceil\}
$$ 
 satisfies \eqref{eq:d}, and we get that \eqref{eq:almost_final} is $O(q^{2r-i+(2k+8)/3})$. Taking the summation over $i=1,\dots, r$, 
 we get that the number of $a$ such that $Q_a$ has rank less than~$r$ is $O(q^{2r+(2k+5)/3})$
 and the result follows in this case.

If $k> (3n+1)/4$, we can assume, that $r\leq n-k$ as otherwise the bound is trivial. Then, the choice $d=n-k-i+1$ satisfies \eqref{eq:d}  and we get that \eqref{eq:almost_final} is $O(q^{2r+2k-n+2-i})$.
Taking the summation over $i=1,\dots, r$, 
 we get that the number of $a$ such that $Q_a$ has rank less than~$r$ is $O(q^{2r+2k-n+1})$
 and the result follows in this case.
\end{proof}

In the next lemma we investigate the rank of the quadratic form $h\mapsto S(ah)-S(bh)$ with $h\in\cG(n-k)$.

\begin{lemma}\label{lemma:num_of_solution_basic}
Let $a,b\in \cM(k)$ be polynomials such that
\begin{equation}\label{eq:star_assumption}
a^*\cdot a \neq b^*\cdot b.    
\end{equation}
Then, the quadratic form $h\mapsto S(ah)-S(bh)$ with $h\in\cG(n-k)$ has rank at least $n-2k-1$.
  \end{lemma}

\begin{proof}
We can assume, that $k<n/2$ as otherwise the bound on the rank is trivial.

Let $B$ be symmetric bilinear form defined by
\begin{equation*}
    B(t^i,t^j)=
\left\{
\begin{array}{cl}
  1/2   & \text{if } |i-j|=1, \\
  0   &  \text{otherwise},
\end{array}
\right. \quad 0\leq i,j\leq n.
\end{equation*}
which
satisfies $S(h)=B(h,h)$. 

First consider the symmetric bilinear form $B_a$ defined by $$
B_a(h,g)=B(ah,ag),
$$ 
with $h,g\in\cG(n-k)$. Then for $0\leq i,j\leq n-k$ we have
$$
B_a(t^i,t^j)=B(a\cdot t^i,a\cdot t^j)=\frac{1}{2} \left(S^{(|i-j-1|)}(a)+S^{(|j-i-1|)}(a) \right).
$$
Similarly, if $B_b$ is the symmetric bilinear form defined by $B_b(h,g)=B(bh,bg)$, $h,g\in\cG(n-k)$, we also have
$$
B_b(t^i,t^j)=B(b\cdot t^i,b\cdot t^j)=\frac{1}{2} \left(S^{(|i-j-1|)}(b)+S^{(|j-i-1|)}(b) \right).
$$
Whence, the symmetric bilinear form $B_{a,b}$ corresponding to $S(ah)-S(bh)$ satisfies
\begin{equation}\label{eq:Bab}
B_{a,b}(t^i,t^j)=\frac{1}{2} \left(S^{(|i-j-1|)}(a)+S^{(|j-i-1|)}(a)-S^{(|i-j-1|)}(b)-S^{(|j-i-1|)}(b) \right).    
\end{equation}

Let $M\in \Fq^{(n-k+1)\times (n-k+1)}$ be the symmetric matrix corresponding to $B_{a,b}$ with respect to the basis $1,t,t^2,\dots, t^{n-k}$. We show, that it has rank at least $n-2k-1$.  To do so
let
$\ell$ be the smallest integer such that $t^\ell \nmid a^*\cdot a - b^*\cdot b$. By \eqref{eq:star_assumption}, $\ell$ is well-defined. 
Also, as $a^*\cdot a - b^*\cdot b$ is self-reciprocal, $\ell \leq k$. Then, by the definition of~$\ell$ and by Lemma~\ref{lemma:star_coeff} we have
\begin{equation}\label{eq:S!=0}
S^{(\nu)}(a)=S^{(\nu)}(b) \quad \text{for } k-\ell < \nu \leq k \quad \text{and} \quad  S^{(k-\ell)}(a)\neq S^{(k-\ell)}(b).
\end{equation}
By definition of $S^{(\nu)}$, we also have, that 
\begin{equation}\label{eq:S=0}
    S^{(\nu)}(a)=S^{(\nu)}(b)=0 \quad \text{for }\nu >k.
\end{equation}

Then the submatrix
$$
(M_{i,j}) \quad \text{with } k-\ell +2\leq i\leq n-k+1 , \quad 1 \leq j\leq  n-2k+\ell
$$
of $M$ is an upper triangular matrix with nonzero elements in the diagonal by \eqref{eq:Bab} and \eqref{eq:S!=0}. Indeed, for $0\leq\nu,\mu<n-2k+\ell$ with $\nu>\mu$, we have
\begin{align*}
M_{k-\ell+2+\nu, 1+\mu}=
\frac{1}{2} \Big(&
S^{(|k-\ell +\nu-\mu |)}(a)+S^{(|\mu-\nu -k+\ell-2|)}(a)\\
& \quad -S^{(|k-\ell +\nu-\mu|)}(b)-S^{(|\mu-\nu -k+\ell-2|)}(b) 
\Big)=0  
\end{align*}
by \eqref{eq:S!=0} and \eqref{eq:S=0}
as $|k-\ell +\nu-\mu|>k-\ell$ and $|\mu-\nu -k+\ell-2|>k-\ell +2$. Moreover, the values in the diagonal are  
\begin{align*}
&M_{k-\ell+2+\nu, 1+\nu}\\
&=
\frac{1}{2} \left(
S^{(|k-\ell|)}(a)+S^{(|-k+\ell-2|)}(a)-S^{(|k-\ell|)}(b)-S^{(|-k+\ell-2|)}(b) 
\right)\\
&=\frac{1}{2} \left(S^{(k-\ell)}(a)-S^{(k-\ell)}(b) \right)\neq 0    
\end{align*}
by \eqref{eq:S!=0} and \eqref{eq:S=0} as $|-k+\ell -2|>k-\ell$.

Thus it has rank $n-2k+\ell-1 $ and so
$$
\rk M\geq n-2k+\ell -1\geq n-2k-1.
$$
which proves the result.
\end{proof}

\subsection{Vaughan's identity for function fields}\label{sec:vau}
In this section, we establish a variant of Vaughan’s identity for function fields, following the method of Car and Mauduit \cite{CM}, but adapted to the present setting. Since the precise form needed here differs from that in \cite{CM}, we include the full proof for completeness.

To do that, let us define the \emph{von Mangoldt function}
$\Lambda$ in $\F_q[t]$ by 
$$
\Lambda (f)=
\left\{
\begin{array}{cl}
 \deg P  & \text{if $f=P^r$ for some irreducible polynomial $P$;}  \\
  0   &  \text{otherwise.}
\end{array}
\right.
$$
We have the following result.

\begin{lemma}\label{lemma_Vau}
Let $\Psi:\F_q[t]\rightarrow \CC$ with $|\Psi(f)|\leq 1$. If $1\leq u,v\leq n$ two integers with $u+v<n$, then
$$
\sum_{f\in\cM(n)} \Lambda(f) \Psi(f)\ll n \Sigma_1 + n^{5/2} q^{n-(u+v)/2} \Sigma_2^{1/2}
$$
with
\begin{align*}
    \Sigma_1&=\sum_{|g|\leq q^{u+v}}\left|\sum_{|h|=q^{n}/|g|}\Psi(gh)\right|\\
    \Sigma_2&=\max_{v\leq i\leq n-u} \max_{|g_1|=q^{n-i}}\sum_{|g_2|=q^{n-i} }\left|\sum_{|h|=q^i}\Psi(hg_1)\overline{\Psi(hg_2)} \right|,
\end{align*}
where the summations are taken over \emph{monic} polynomials and the implied constant is absolute.
\end{lemma}

\begin{proof}
Let $\mu$ be the \emph{Möbius function} on $\F_q[t]$ defined by
$$
\mu (f)=
\left\{
\begin{array}{cl}
 (-1)^k  & \text{if $f=P_1\dots P_k$ for some distinct }  \\
   & \text{\quad irreducible polynomials $P_1,\dots, P_k$;}  \\
  0   &  \text{otherwise.}
\end{array}
\right.
$$
Then, one can write 
\begin{equation*}
    \Lambda(f)=\sum_{a\mid f}\mu(a)\deg (f/a),
\end{equation*}
where the sum is taken over only \emph{monic} polynomials. 
We split this sum according to $|a|$. We keep the terms $|a|\leq q^u$, and the we rewrite the remaining terms as
$$
\sum_{\substack{a\mid f\\ |a|>q^u}}\mu(a)\deg (f/a)=\underset{\substack{ab\mid f\\ |a|>q^u}}{\sum \sum} \mu(a)\Lambda(b).
$$
Finally, we keep the terms $|b|>q^v$ and rewrite the remaining terms as
$$
\underset{\substack{ab\mid f\\ |a|>q^u, |b|\leq q^v}}{\sum \sum} \mu(a)\Lambda(b)=
\underset{\substack{ab\mid f\\  |b|\leq q^v}}{\sum \sum} \mu(a)\Lambda(b)-
\underset{\substack{ab\mid f\\ |a|\leq q^u, |b|\leq q^v}}{\sum \sum} \mu(a)\Lambda(b).
$$
Here the first term is a complete sum for $a$ and thus vanishes unless $b=f$ which is not possible as $|b|\leq q^v<|f|$. Whence, we get
\begin{equation*}
    \Lambda(f)= \sum_{\substack{a\mid f \\ |a|\leq q^u}} \mu(a) \deg (f/a) - \underset{\substack{ab\mid f\\ |a|\leq q^u, |b|\leq q^v}}{\sum  \ \sum } \mu(a)\Lambda(b) + \underset{\substack{ab\mid f\\ |a|> q^u, |b|>q^v}}{\sum  \ \sum } \mu(a)\Lambda(b).
\end{equation*}
It yields
\begin{equation*}
\sum_{f\in\cM(n)}\Lambda(f)\Psi(f)
= S_1-S_2+S_3,
\end{equation*}
where
\begin{align*}
S_1&=\underset{\substack{|ab|=q^n\\ |a|\leq q^u}}{\sum \ \sum}
 \mu(a) \deg (b) \Psi(ab),\\
 S_2&=\underset{\substack{|abc|=q^n\\ |a|\leq q^u, |b|\leq q^v}}{\sum  \ \sum \ \sum} \mu(a)\Lambda(b)\Psi(abc) ,\\
 S_3&=\underset{\substack{|abc|=q^n\\ |a|> q^u, |b|>q^v}}{\sum  \ \sum \sum } \mu(a)\Lambda(b)\Psi(abc).
\end{align*}
For $S_1$, we have
\begin{equation*}
S_1=\sum_{|a|\leq q^u}\mu(a)(n-\deg (a))\sum_{\substack{|b|=q^n/|a|}}\Psi(ab)\ll n\sum_{|a|\leq q^u}\left|\sum_{\substack{|b|=q^n/|a|}}\Psi(ab)\right|.
\end{equation*}
For $S_2$, observe that 
$$
\sum_{b\mid z}\Lambda(b)=\deg (z).
$$
Then, we have 
\begin{equation*}
\begin{split}
    S_2&=\sum_{r=n-u-v}^n \left(  \underset{\substack{|a|\leq q^u \, |b|\leq q^v\\ |ab|=q^{n-r}}}{\sum \ \sum}\mu(a)\Lambda(b)\right)\sum_{|c|=q^r}\Psi(abc)\\
&
\ll
\sum_{r=n-u-v}^n
\sum_{|z|=q^{n-r}}\left(\sum_{\substack{b\mid z\\ |b|\leq q^v}} \Lambda(b)\right)
\left|\sum_{|c|=q^n/|z|}\Psi(zc)\right|
    \\
    &\ll (u+v)\sum_{|z|\leq q^{u+v}}\left|\sum_{|c|=q^n/|z|}\Psi(zc)\right|.
\end{split}
\end{equation*}

For $S_3$, observe that
$$
\sum_{a\mid c}\mu(a)=\left\{ 
\begin{array}{cl}
  1   & \text{if $c=1$}; \\
  0   & \text{otherwise}.
\end{array}
\right.
$$
Then, we have
\begin{equation*}
    S_3=\sum_{q^v\leq |b|\leq q^{n-u}}\Lambda(b)\sum_{|c|=q^n/|b|}\left(\sum_{\substack{a\mid c\\ |a|>q^u}}\mu(a)\right)\Psi(bc).
\end{equation*}
For $c\in\F_q[t]$, write
$$
\beta_c=\sum_{\substack{a\mid c\\ |a|> q^u}}\mu(a).
$$
Then, by the Cauchy-Schwarz inequality,
$$
\left|S_3\right|^2\leq \left(\sum_{q^v\leq |b|\leq q^{n-u}}\Lambda(b)^2\right)\sum_{q^v\leq |b|\leq q^{n-u}}\left|\sum_{|c|=q^n/|b|}\beta_c\Psi(bc)  \right|^2.
$$

As
$$
\sum_{|f|=q^m}\Lambda(f)=\sum_{d\mid m}\sum_{P \in \cP(d)}d=\sum_{d\mid m}d\pi_q(d)=q^m,
$$
we have
for the first term, that
$$
\sum_{q^v\leq |b|\leq q^{n-u}}\Lambda(b)^2 \leq n \sum_{ |b|\leq q^{n-u}}\Lambda(b)
\ll n q^{n-u} .
$$
For the second term,
\begin{align*}
&\sum_{q^v\leq |b|\leq q^{n-u}}\left|\sum_{|c|=q^n/|b|}\beta_
c\Psi(bc)  \right|^2\\
&=\sum_{v\leq i\leq n-u} \sum_{|c_1|=q^{n-i}}\sum_{|c_2|=q^{n-i}} \beta_{c_1}\overline{\beta_{c_2}}\sum_{|b|=q^i}\Psi(bc_1)\overline{\Psi(bc_2)}\\
&\ll \sum_{v\leq i\leq n-u} \sum_{|c_1|=q^{n-i} } |\beta_{c_1}|^2 \sum_{|c_2|=q^{n-i} }\left|\sum_{|b|=q^i}\Psi(bc_1)\overline{\Psi(bc_2)} \right| \\
&\ll \sum_{v\leq i\leq n-u} \left(\sum_{|c_1|=q^{n-i} } |\beta_{c_1}|^2\right) \max_{|c_1|=q^{n-i}}\sum_{|c_2|=q^{n-i} }\left|\sum_{|b|=q^i}\Psi(bc_1)\overline{\Psi(bc_2)} \right|
\end{align*}
by $|\beta_{c_1}\overline{\beta_{c_2}}|\leq |\beta_{c_1}|^2/2 +|\beta_{c_2}|^2/2$ and by the symmetry.

By Lemma~\ref{lemma:tau_moment},
$$
\sum_{|c_1|=q^{n-i} } |\beta_{c_1}|^2\ll \sum_{|c_1|=q^{n-i} } |\tau(c_1)|^2 \ll n^{3}q^{n-i}.
$$

Thus
$$
|S_3|^2\ll n^5q^{2n-u-v} 
\max_{v\leq i\leq n-u}\max_{|c_1|=q^{n-i}}\sum_{|c_2|=q^{n-i} }\left|\sum_{|b|=q^i}\Psi(bc_1)\overline{\Psi(bc_2)} \right|.
$$
\end{proof}

\section{Proof of Theorem \ref{thm:1}}\label{sec:proof}
Let $\gamma\in \Fq$. By the orthogonality of additive characters, we have
\begin{equation*}
\#\cR_n(\gamma)=\frac{1}{q}\sum_{\psi}\sum_{f\in\cP(n)}\psi(R(f)-\gamma),
\end{equation*}
where the first summation is taken over all additive characters of $\Fq$. For the trivial additive character $\psi_0$, we have
\begin{equation*}
   \frac{1}{q}\sum_{f\in\cP(n)}\psi_0(R(f)-\gamma)=\frac{\#\cP(n)}{q}
\end{equation*}
which gives the main term.

The number of irreducible polynomials $f$ of degree $\deg f<n$ with $\deg f\mid n$ is $O(q^{n/2})$ by \eqref{eq:PNT}, thus for non-trivial additive characters $\psi\neq \psi_0$, we have
\begin{equation*}
  \begin{split}
    \sum_{f\in\cP(n)}\psi(R(f))
    &=\frac{1}{n}\sum_{f\in \cP(n)}\Lambda(f)\psi(R(f))\\
    &=\frac{1}{n}\sum_{f\in\cM(n)}\Lambda(f)\psi(R(f)) +O(q^{n/2}).
    \end{split}
\end{equation*}

Let $1\leq u,v\leq n$ be two integers with
\begin{equation}\label{eq:uv-constr}
   (3n+1)/4< u+v< n
\end{equation}
to be fixed later.
Then,
by Lemma~\ref{lemma_Vau},
\begin{equation}\label{eq:Vau}
\sum_{f\in\cM(n)}\Lambda(f)\psi(R(f)) \ll  n \Sigma_1 + n^{5/2} q^{n-(u+v)/2} \Sigma_2^{1/2},
\end{equation}
where
$$
\Sigma_1=\sum_{|g|\leq q^{u+v}}\left|\sum_{|h|=q^{n}/|g|}\psi(R(gh))\right|
$$
and
$$
\Sigma_2=\max_{v\leq i\leq n-u} \max_{|g_1|=q^{n-i}}\sum_{|g_2|=q^{n-i} }\left|\sum_{|h|=q^i}\psi(R(hg_1))\overline{\psi(R(hg_2))} \right|.
$$

In order to estimate $\Sigma_1$, write
$$
\Sigma_1=\sum_{k\leq u+v}Z_k
$$
with
$$
Z_k=\sum_{|g|=q^k}\left|\sum_{|h|=q^{n}/|g|}\psi(R(hg))\right|.
$$

First, consider the terms $Z_k$ with $k<15n/32 $. For such $k$ we have by Lemma~\ref{lemma:triv-rank} that the rank of the quadratic form $h\mapsto S(gh)$ is at least $n-2k$. If we consider it on the affine space of monic polynomials, its rank is reduced by at most one, thus by \eqref{eq:lin_red} and by Lemma~\ref{lemma:gauss}
\begin{equation}\label{eq:k:small}
Z_k \ll q^k\cdot q^{n-k-(n-2k-1)/2}=q^{n/2 +k+1/2}, \end{equation}
for  $k< 15n/32$.

Next, consider the terms $Z_k$ with $15n/32\leq  k \leq (3n+1)/4$.
Write
$$
\cB_k=\{g\in\cM(k): \rank Q_g < \lfloor 2(k-5)/15\rfloor\},
$$
where $Q_g$ is the quadratic form $h\mapsto S(gh)$. 
We have by Lemma~\ref{lemma:single:bad:a} that
\begin{equation*}
   \#\cB_k\ll q^{ (14k+5)/15  }. 
\end{equation*}
For $g\in\cB_k$, we estimate the inner sum of $Z_k$ trivially by $q^{n-k}$. For $g\not\in \cB_k$, we have $\rk Q_g\geq \lfloor 2(k-5)/15\rfloor\geq (2k-25)/15$. If we consider $Q_g$ on the affine space of monic polynomials, its rank is reduced by at most one. Thus, by
Lemma~\ref{lemma:gauss} we have \begin{equation}\label{eq:k:medium}
Z_k \ll q^{ (14k+5)/15  }\cdot q^{n-k}+q^k\cdot q^{n-k-(k-20)/15}\ll q^{n-k/15+4/3}
, \quad 
\end{equation}
for  $15n/32 \leq k\leq  (3n+1)/4$. 

Finally, for $(3n+1)/4< k\leq u+v$ we write
$$
\cB_k=\{g\in\cM(k): \rank Q_g < \lfloor 2(n-k-1)/5 \rfloor\}.
$$
We have by Lemma~\ref{lemma:single:bad:a} that
\begin{equation*}
   \#\cB_k\ll q^{ (6k-n+1)/5}. 
\end{equation*}
For $g\in\cB_k$, we estimate the inner sum of $Z_k$ trivially by $q^{n-k}$. For $g\not\in \cB_k$, $\rk Q_g\geq \lfloor 2(n-k-1)/5 \rfloor\geq 2(n-k-6)/5  \}$. If we consider $Q_g$ on the affine space of monic polynomials, its rank is reduced by at most one. Thus, by
Lemma~\ref{lemma:gauss} we have 
\begin{equation}\label{eq:k:large}
Z_k \ll q^{ (6k-n+1)/5}\cdot q^{n-k}+q^k\cdot q^{n-k-(n-k-11)/5 }\ll q^{(4n+k+11)/5}
, \quad 
\end{equation}
for  $(3n+1)/4 <k\leq u+v$. 

Then \eqref{eq:k:small}, \eqref{eq:k:medium} and \eqref{eq:k:large} yield
\begin{align}\label{eq:Sigma_1:final}
    \Sigma_1   
    &\ll \sum_{k< 15n/32 }Z_k + \sum_{15n/32\leq k\leq (3n+1)/4} Z_k + \sum_{(3n+1)/4< k\leq u+v} Z_k
        \\ \notag
    &\ll \sum_{k< 15n/32 }q^{n/2+k+1/2} + \sum_{15n/32
    \leq k\leq (3n+1)/4} q^{n-k/15+4/3}\\ \notag
    & \quad+ \sum_{(3n+1)/4< k\leq u+v} q^{(4n+k+11)/5    }
    \\
    &\ll q^{31n/32+4/3 } + q^{(4n+u+v+11)/5} 
. \notag
\end{align}

In order to estimate $\Sigma_2$, let us fix $v\leq i\leq n-u$ and consider
$$
I(g_1,g_2)=\sum_{|h|=q^i}\psi(R(g_1h)-R(g_2h))
$$
with monic polynomials $g_1,g_2$ of degree $n-i$.
For given $g_1$, let 
$$
\cB(g_1)=\{g\in \cM(n-i): g^*g=g_1^*g_1 \}.
$$
If $g_2\not\in \cB(g_1)$, we have by Lemma~\ref{lemma:num_of_solution_basic}, that the rank of the quadratic form $h\mapsto S(g_1h)-S(g_2h)$ with $h\in \cG(i)$ is at least $n-2(n-i)-1=2i-n-1$. If we consider it only on monic polynomials $h\in\cM(i)$,
then its rank is reduced by at most one. As $R(g_1h)-R(g_2h)=
S(g_1h)-S(g_2h)+L_{g_1,g_2}(h)$ with some linear polynomial $L_{g_1,g_2}$ in the coefficients of $h$ by Lemma~\ref{lemma:gauss} we get
$$
I(g_1,g_2)\ll q^{i-(2i-n-2)/2}=q^{n/2+1}.
$$
If $g_2\in\cB(g_1)$, we use the trivial bound
$$
I(g_1,g_2)\ll q^i.
$$

By Lemma~\ref{lemma:quadratic_equation} we have
$$
\#\cB(g_1)\ll \min\left\{ 2^n,q^{(2+\varepsilon)\frac{n}{\log n}}\right\}
$$
for any $\varepsilon >0$, where the implied constant may depend on $\varepsilon$.
Whence, for $0< \delta <1/2$ we have
$$
\#\cB(g_1)\ll  q^{ \delta n},
$$ 
where the implied constant may depend on $\delta$. Then
\begin{equation}\label{eq:Sigma2}
\Sigma_2\ll \max_{v\leq i\leq n-u}\left(q^{\delta n}q^i +q^{n-i}q^{n/2+1}  \right)\ll q^{(1+\delta)n-u}+q^{3n/2 -v+1}.
\end{equation}
Combining \eqref{eq:Vau}, \eqref{eq:Sigma_1:final} and \eqref{eq:Sigma2}, we get
\begin{align*}
\sum_{f\in\cM(n)}\Lambda(f)\psi(R(f)) &\ll n\left(q^{31n/32+4/3 } + q^{
(4n+u+v+11)/5} \right) \\&\quad 
+ n^{5/2}\left( q^{(3/2 + \delta/2)n-u-v/2}+q^{7n/4-u/2-v+1/2}\right).
\end{align*}
If we choose
$$
u=\left(\frac{7}{38}+\frac{12}{19}\delta\right) n \mand v=\left(\frac{13}{19}-\frac{7}{19}\delta\right)n,
$$
then \eqref{eq:uv-constr} is fulfilled and
we get
\begin{align*}
\sum_{f\in\cM(n)}\Lambda(f)\psi(R(f)) 
&\ll   n^{5/2}q^{(37/38 + \delta /19)n +11/5}.
\end{align*}
Replacing $\delta$ by $19\delta$, the theorem follows.


\section*{Acknowledgment}

The author wishes to thank Arne Winterhof for his valuable comments. The author also  thanks the anonymous referees for their careful reading and constructive comments, which led to several improvements in the presentation and in the clarity of the arguments.

The author was partially supported by NRDI (National Research Development and Innovation Office, Hungary) grant FK 142960 and by the János Bolyai Research Scholarship of the Hungarian Academy
of Sciences.


\end{document}